\documentclass[a4paper]{amsart}
\usepackage[utf8]{inputenc}
\usepackage{amsthm}
\usepackage{amsmath}
\usepackage{amssymb}
\usepackage{mathrsfs}

\usepackage[top=3cm, bottom=3.5cm, left=3cm, right=3cm]{geometry}
\usepackage{comment}

\usepackage{graphicx}
\graphicspath{ {./images/} }

\usepackage{mathtools}

\usepackage{tikz}
\usepackage{tikz-3dplot}
\usetikzlibrary{3d}

\usepackage{algorithm}
\usepackage{algpseudocode}

\usepackage[graphicx]{realboxes}
\usepackage{caption}

\usepackage[plainpages=false,pdfpagelabels,colorlinks=true,citecolor=blue,hypertexnames=false]{hyperref}
\usepackage[capitalise]{cleveref}

\newtheorem{thm}{Theorem}
\newtheorem{prop}[thm]{Proposition}
\newtheorem{conj}[thm]{Conjecture}
\newtheorem{lem}[thm]{Lemma}
\newtheorem{dfn}[thm]{Definition}

\renewcommand{\phi}{\varphi}
\renewcommand{\epsilon}{\varepsilon}

\theoremstyle{remark}

\title{On Nieuwland Numbers and Polar Duality}
\author{Kavin Satheeskumar\, and \, Liam Benoit}

\thanks{\textit{Markham, Ontario}, \textit{Canada}. \href{mailto:ksathees@uwaterloo.ca}{ksathees@uwaterloo.ca}} 

\thanks{\textit{Kelowna, British Columbia}, \textit{Canada}. \href{mailto:tionebl@student.ubc.ca}{tionebl@student.ubc.ca}}

\date{\today}

\begin{document}

\begin{abstract}
	\noindent
	A convex 3D-polytope is said to have Rupert's property if it can pass through a copy of itself. The Nieuwland number of a convex polytope $P$ is the largest $\nu \in \mathbb{R}^+$ such that $\nu P$ can pass through $P$. We reduce showing $P$ passes through $Q$ to a feasibility problem over a quadratic constraint set. Using this, we prove that the Nieuwland number of the octahedron is $\frac{3\sqrt2}{4}$ and that the computation of the Nieuwland number of a convex polytope can be reduced to polynomially many semialgebraic optimization problems in a fixed number of variables.
\end{abstract}



\maketitle

\section{Introduction} \label{sec:intro}

More than 300 years ago, Prince Rupert was the first to show that it was possible to bore a hole into a unit cube that allows an identical unit cube to pass through. In 1816, the Dutch scientist Pieter Nieuwland showed that a slightly larger cube with side length $\frac{3\sqrt{2}}{4}$ could pass through a unit cube \cite{BeGuHuJo21}. This is the largest side length for which a cube can pass through a unit cube. This constant would later become known as the Nieuwland number of the cube.\\\\
We say that a convex polytope $ P \subset \mathbb{R}^3$ has Rupert's property if, like the cube, it is possible to bore a hole into it such that an identical copy of the original shape can pass through. Similarly, we define the Nieuwland number of $P$ as the largest $\nu \in \mathbb{R}^+$ such that $\nu P$ can pass through $P$. These definitions are rigorously stated in \ref{subsec:definitions}.\\\\
Over the centuries, various other solids were shown to have Rupert's property, including the octahedron and tetrahedron in 1968 \cite{Scriba68}. In the last 10 years, there has been an explosion of progress as computational methods have been applied to the problem (\cite{Fe23}, \cite{StYu23} and \cite{StYu25}). As such, better and better estimates of the Nieuwland number of various convex solids were discovered.\\\\
Despite recent advances in the computation of Nieuwland numbers, the cube remains the only shape with a non-trivial Nieuwland number whose Nieuwland number is known exactly. There are also very few theorems which prove anything about Nieuwland numbers in general or the relationships between them. However, numerical efforts have allowed for various conjectures to be posed. The most notable of these is the following:

\begin{conj}[Steininger - Yurkevich \cite{StYu23}] \label{conj:cube_octo}
	The cube and the octahedron have the same Nieuwland number.
\end{conj}
\noindent
In this paper, we prove this conjecture true by proving Theorem \ref{thm:oct_nieuwland}. In Section \ref{sec:preliminaries} we introduce the necessary definitions and concepts. In particular, we introduce Definition \ref{def:pass_through} that gives us our definition of the Nieuwland number. We then prove that it is equivalent to the definition of the Nieuwland number used in \cite{StYu23}. 
\\\\
In Section \ref{sec:dual_pass_through_condition} we develop a condition on polytopes $P$ and $Q$ with $P$ passing through $Q$. When this condition is met, it can be shown that $Q^*$ passes through $P^*$, where $P^*$ and $Q^*$ are the polar duals of $P$ and $Q$ respectively. We show that this condition is met when $Q$ is an octahedron and $P=\nu Q$ for any $\nu$ such that $P$ passes through $Q$. Similarly, this condition is satisfied when $Q$ is the cube and $P=\frac{3\sqrt2}{4}Q$. Combining these two results gives a proof of Theorem \ref{thm:oct_nieuwland}. 
\\\\
In Section \ref{sec:main_algorithm}, we use the ideas from the previous two sections to develop algorithm \ref{alg:Nieuwland}, a polynomial time algorithm for computing Nieuwland numbers. We do this by proving Theorem \ref{thm:pass_through_simplified}, which gives us another alternative equivalent condition for $P$ passing through $Q$. This condition, along with Lemma \ref{lem:hyperplanes}, naturally produces a polynomial time algorithm.\\\\
In Section \ref{sec:conclusion}, we discuss potential problems for future research.
\\
\subsection{Acknowledgments}
\hfill \\
The authors thank the University of Waterloo, Professor Ricardo Fukasawa, and Professor Daniel Bienstock for providing us with the computational resources, advice, and opinions needed to complete this paper. Moreover, the authors also thank Kareem Alfarra and Alex Pawelko for their feedback on previous drafts as well as Professor Andriy Prymak for his feedback and endorsement for submission to arXiv. Finally, the authors thank Thomas Walter Murphy VII for originally bringing this problem to their attention.

\section{Preliminaries} \label{sec:preliminaries}
For the remainder of this paper, we will let $P$ and $Q$ be convex polytopes in $\mathbb{R}^3$, $U\in SO(3),\eta\in\mathbb{R}^3,\delta\in\mathbb{R}^3$. We work over $\mathbb{R}^3$ with the standard topology.
\\
\subsection{Halfspace Representations and Polar Duality}
\hfill \\
Convex polytopes are, by definition, the convex hull of finitely many points with non-empty interior. Any convex polytope can also be represented as the intersection of finitely many halfspaces. We can write $P=\bigcap_{i=1}^n\{x\in\mathbb{R}^3: a_i^Tx\leq b_i\}$ which is often abbreviated to $P=\{Ax\leq b\}$, where the $i$th row of $A$ is $a_i^T$ and $b=(b_1,...,b_n)^T$. When $0\in \textrm{int}(P)$ we have $b_i>0$. We can then write 
$$\{x\in\mathbb{R}^3: \forall 1 \leq i \leq n, \quad a_i^Tx\leq b_i \}$$  
$$=\left\{x\in\mathbb{R}^3: \forall 1 \leq i \leq n, \quad \left(\frac{1}{b_i}a_i^T \right)x\leq 1 \right\}.$$
Setting $a_i'^T=\frac{1}{b_i}a_i^T$ gives $P=\{A'x\leq 1\}$. This form is especially important when dealing with polar duals.
\\\\
In convex geometry, the polar dual of a convex polytope is another convex polytope, whose faces are the vertices of the original and whose vertices are the faces of the original. Formally, if $0\in \textrm{int}(P)$ we define the polar dual
$$P^*=\{x\in\mathbb{R}^3:x^Ty\leq1,\forall y\in P\}.$$
This transformation has many nice properties, such as $P^*$ being a convex polytope with $0\in \textrm{int}(P^*)$ so long as $P$ is a convex polytope with $0\in \textrm{int}(P)$. The fact that the polar dual swaps the vertices and faces can be stated formally as follows: $P=\{Ax\leq1\}=\textrm{conv}\{v_1,...,v_n\}$ if and only if $P^*=\textrm{conv}\{a_1,...,a_m\}=\{Vx\leq1\}$, where $a_i^T$ is the $i$th row of $A$ and $v_k^T$ is the $k$th row of $V$. From this it can be shown that $P^{**}=P$. Proofs of these facts can be found in \cite{Ga07}.
\\\\
Other properties of the dual that will be used in this paper include:
\begin{itemize}
	\item $P\subseteq Q$ if and only if $Q^*\subseteq P^*$.
	\item For $c \in \mathbb{R}$ with $c > 0$, $(cP)^*=\frac{1}{c}P^*$.
	\item If $P$ is point symmetric (i.e. $P=-P$), then so is $P^*$.
\end{itemize}
\hfill
\subsection{Rupert's Property and Nieuwland Numbers} \label{subsec:definitions}
\hfill \\
In order to define Rupert's property and Nieuwland numbers rigorously, a definition of ``$P$ passes through $Q$'' is needed. There are many competing but equivalent definitions for $P$ passing through $Q$; we will use the following.
\begin{dfn} \label{def:pass_through}
	$P$ \textbf{passes through} $Q$ if there exists $(U,\eta,\delta)$ such that $$\textup{proj}_{\eta}(UP+\delta)\subseteq \textup{proj}_{\eta}(Q),$$where $\textup{proj}_{\eta}$ is the orthogonal projection onto the hyperplane orthogonal to $\eta$. When $\eta=0$ we let $\textup{proj}_\eta$ be the identity. We say $P$ passes through $Q$ with \textbf{certificate} $(U,\eta,\delta)$ if $\textup{proj}_{\eta}(UP+\delta)\subseteq \textup{proj}_{\eta}(Q)$.\\
\end{dfn}
\noindent
The following definition is equivalent to \ref{def:pass_through} and will be used interchangeably throughout the paper.
\begin{prop} \label{proposition:equivalent_pass_through_definition}
	Let $P=\textup{conv}\{v_1,...,v_n\}$ and $Q=\{Ax\leq b\}$. $P$ passes through $Q$ with certificate $(U,\eta,\delta)$ if and only if there exist $z_1,...,z_n\in\mathbb{R}$ such that
	$$\forall 1\leq k\leq n, \quad A(Uv_k+z_k\eta+\delta)\leq b.$$
\end{prop}
\begin{proof}
	Let $P, Q$ be as in the statement of the proposition. We note that if $P$ passes through $Q$ with certificate $(U, \eta, \delta)$ we can assume that $\|\eta\|\in\{0,1\}$ since $\textrm{proj}_\eta=\textrm{proj}_{r\eta}$ for all $r>0$. Similarly, if $\forall 1\leq k\leq n, \quad A(Uv_k+z_k\eta+\delta)\leq b$ we can assume $\|\eta\|\in\{0,1\}$ by taking $\eta'=\frac{\eta}{\|\eta\|}$ and $z_k'=\|\eta\|z_k$ when $\|\eta\|>0$. This assumption allows us to write $\textrm{proj}_\eta(x)=x-(x^T\eta)\eta$.
	\\\\
	$(\Longrightarrow)$
	\\\\
	Since $P$ passes through $Q$ there exists $(U,\eta,\delta)$ such that $\textrm{proj}_\eta(UP+\delta)\subseteq \textrm{proj}_\eta(Q)$. Then, for all $1\leq k\leq n$, there exists $x\in Q$ such that
	$$\textrm{proj}_\eta(Uv_k + \delta)=\textrm{proj}_\eta(x).$$
    Recall that $\textrm{proj}_\eta(x)=x-(x^T\eta)\eta$. This gives us
	$$(Uv_k+\delta)-((Uv_k+\delta)^T\eta)\eta=x-(x^T\eta)\eta,$$
	$$(Uv_k+\delta)-((Uv_k+\delta-x)^T\eta)\eta=x.$$
	Let $z_k=-(Uv_k+\delta-x)^T\eta$. Then $Uv_k+\delta+z_k\eta=x\in Q=\{Ax \leq b\}$. So, 
	$$A(Uv_k+z_k\eta+\delta)\leq b.$$
    As needed.
    \\\\
	$(\Longleftarrow)$
	\\\\
    We have $\forall 1 \leq k \leq n, \quad A(Uv_k+z_k \eta + \delta) \leq b$ and we want to show that $\textup{proj}_{\eta}(UP+\delta)\subseteq \textup{proj}_{\eta}(Q)$. We let $y\in \textrm{proj}_\eta(UP+\delta)$ be arbitrary and show that it is always possible to find some $C$ such that $y+C\eta \in Q$.\\\\
	Let $y\in \textrm{proj}_\eta(UP+\delta)$. Then, for some $x\in P, y=(Ux+\delta)-((Ux+\delta)^T\eta)\eta$. Since $x\in P,x=c_1v_1+...+c_nv_n$ for some $c_1,...,c_n\geq0$ with $c_1+...+c_n=1$. Let $z=(Ux+\delta)^T\eta$. Then,
	$$A\left(Ux+\delta+\sum_{k=1}^nc_kz_k\eta\right)=\sum_{k=1}^nc_kA(Uv_k+z_k\eta+\delta) \leq \sum_{k=1}^nc_kb=b.$$
	Thus, 
	$$y+\left(\sum_{k=1}^nc_kz_k+z\right)\eta=Ux+\delta+\sum_{k=1}^nc_kz_k\eta\in Q.$$
	Finally,
	$$y=\textrm{proj}_\eta\left(y+\left(\sum_{k=1}^nc_kz_k+z\right)\eta\right)$$ so
	$y\in \textrm{proj}_\eta(Q)$ as needed.
\end{proof}
\noindent
With a definition of what it means for $P$ to pass through $Q$ we can now define Rupert's property and Nieuwland numbers. 
\begin{dfn}\label{dfn:Nieuwland 1}
	$P$ has \textbf{Rupert's property} if $\nu P$ passes through $P$ for some $\nu>1$. The \textbf{Nieuwland number} of $P$, denoted $\nu_P$, is the supremum over all $\nu\geq1$ such that $\nu P$ passes through $P$.
\end{dfn}
\noindent
Below we show that we can replace supremum with the word maximum in the above definition of the Nieuwland number.
\begin{lem}\label{lem:max_number}
Let $\nu_P$ be the Nieuwland number of $P$. Then, $\nu_P P$ passes through $P$.
\begin{proof}
Let $\nu_n$ be a sequence converging to $\nu_P$ such that $\nu_n P$ passes through $P$ with certificate $(U_n,\eta_n,\delta_n)$. Since $P$ is compact, there exists $M>0$ such that $\|x\|\leq M$ for all $x\in P$. As in the proof of Proposition \ref{proposition:equivalent_pass_through_definition}, we may assume $\|\eta_n\|\in\{0,1\}$ for all $n\in\mathbb{N}$. 
\\\\
We start with the case where there exists a subsequence of $\eta_n$ converging to $0$ (i.e. $\|\eta_n\|\to0$). WLOG, we will assume that the sequence $\eta_n$ converges to $0$. Since $\|\eta\|\in\{0,1\}$ this means $\eta_n$ is eventually $0$. Again, by reducing to a subsequence, we can assume $\eta_n=0$ for all $n\in\mathbb{N}$. This gives $U_n\nu_nP+\delta_n\subseteq P$. Since $\det(U)=1$ and $U_n\nu_nP+\delta_n\subseteq P$ we have $$\text{vol}(U_n\nu_nP+\delta_n)=\nu_n^3\text{vol}(P)\leq \text{vol}(P),$$ where vol is the standard measure on $\mathbb{R}^3$. This implies $\nu_n \leq 1$. Setting $\nu=1, U=Id, \delta=0, \eta=0$ we have $P\subseteq P$. Thus, $\nu_n\to \nu_P\leq 1$ and $\nu_P \geq 1$. So, $\nu_P=1$ and $\nu_P P$ passes through $P$ as desired in this case.
\\\\
In the case where no such subsequence exists, it must be the case that $\|\eta_n\|=1$ eventually. By reducing to a subsequence, we can assume $\|\eta_n\|=1$ for all $n\in \mathbb{N}$. Since $S^2$ is compact, we can further assume that $\eta_n$ converges to $\eta$. We note that $\textrm{proj}_{\eta_n}(U_n\nu_nP+\delta_n)=\textrm{proj}_{\eta_n}(U_n\nu_nP+\textrm{proj}_{\eta_n}(\delta_n))$. So, we can assume $\langle \eta_n,\delta_n\rangle=0$ for all $n\in\mathbb{N}$. This gives $\textrm{proj}_{\eta_n}(U_n\nu_nP+\delta_n)=\textrm{proj}_{\eta_n}(U_n\nu_nP)+\delta_n$. Note that projections are non-expansive, specifically $\|\textrm{proj}_{\eta_n}(x)\|\leq \|x\|$. So, $y\in \textrm{proj}_{\eta_n}(P)$ implies $\|y\|\leq M$. Since $$\textrm{proj}_{\eta_n}(U_n\nu_nP)+\delta_n\subseteq \textrm{proj}_{\eta_n}(U_n\nu_nP+\delta_n)\subseteq \textrm{proj}_{\eta_n}(P),$$ $x\in P$ implies, $$\|\textrm{proj}_{\eta_n}(U_n\nu_nx)+\delta_n\|\leq M.$$ By the reverse triangle inequality,
$$\|\delta_n\|\leq \|\textrm{proj}_{\eta_n}(U_n\nu_nx)+\delta_n\| +\|\textrm{proj}_{\eta_n}(U_n\nu_nx)\|\leq (1+\nu_n)M\leq (1+\nu_P)M.$$ As in the first case, the $\delta_n$ are bounded and $SO(3)$ is compact, so we may assume $\delta_n\to\delta, U_n\to U$.
\\\\
Since $\textrm{proj}_{\eta}(U\nu P+\delta)$ and $\textrm{proj}_{\eta}(P)$ are both convex, it suffices to prove that the vertices of $\textrm{proj}_{\eta}(U\nu P+\delta)$ can be written as a convex combination of the vertices of $\textrm{proj}_{\eta}(P)$. The vertices of $\textrm{proj}_{\eta}(U\nu P+\delta)$ (resp. $\textrm{proj}_{\eta}(P)$) are a subset of $\{\textrm{proj}_{\eta}(U\nu v_i+\delta): \forall 1\leq i\leq m\}$ (resp. $\{\textrm{proj}_{\eta}(v_i): \forall 1\leq i\leq m\}$, where $\{v_i:\forall 1\leq i \leq m\}$ is the set of vertices of $P$. Thus, it suffices to show that all $\{\textrm{proj}_{\eta}(U\nu v_i+\delta): \forall 1\leq i\leq m\}$ can be written as a convex combination of $\{\textrm{proj}_{\eta}(v_i): \forall 1\leq i\leq m\}$.
\\\\
Fix $1\leq j\leq m$. Since $\textrm{proj}_{\eta_n}(U_n\nu_nP+\delta_n)\subseteq \textrm{proj}_{\eta_n}(P)$ we can find $c_{1,n},...,c_{m,n}\geq0, c_{1,n}+...+c_{m,n}=1$ such that $$\textrm{proj}_{\eta_n}(U_n\nu_nv_j+\delta_n)=c_{1,n}\textrm{proj}_{\eta_n}(v_1)+...+c_{m,n}\textrm{proj}_{\eta_n}(v_m).$$ By reducing to a subsequence, we can assume $c_{i,n}\to c_i$ for each $1\leq i \leq m$. This maintains the property that $c_1,...,c_m\geq 0, c_1+...+c_m=1$. Finally, by taking limits on each side we obtain $$\textrm{proj}_{\eta}(U\nu_Pv_j+\delta)=c_1\textrm{proj}_{\eta}(v_1)+...+c_m\textrm{proj}_{\eta}(v_m).$$ Thus, $\textrm{proj}_{\eta}(U\nu_PP+\delta)\subseteq \textrm{proj}_\eta(P)$ as desired.
\end{proof}
\end{lem}
\noindent
These definitions may differ from others in the literature. In \cite{StYu23}, Steininger and Yurkevich provide a definition of Rupert's property and Nieuwland numbers using spherical coordinates. This is useful in their context as it allows them to reduce the dimension of the space they are working over when calculating Nieuwland numbers.
\begin{dfn}
	Let $X(\theta,\phi)=(\cos(\theta)\sin(\phi),\sin(\theta)\sin(\phi),\cos(\phi))$ for $(\theta,\phi)\in[0,2\pi)\times[0,\pi]$ be a parametrization of the 2-sphere. The projection onto the plane orthogonal to $X(\theta,\phi)$ is defined by 
		$$M_{\theta,\phi}=\begin{bmatrix}
		-\sin(\theta) & \cos(\theta) & 0\\
		-\cos(\theta)\cos(\phi) & -\sin(\theta)\cos(\phi) & \sin(\phi)
		\end{bmatrix}.
		$$
	The translation $T_{x,y}:\mathbb{R}^2\to\mathbb{R}^2$ is defined by $$T_{x,y}(a,b)^T=(a+x,b+y)^T.$$
	Finally, for $\alpha\in [0,2\pi)$ the rotation by $\alpha$ is defined as 
        $$R_\alpha =\begin{bmatrix}
        \cos(\alpha) & -\sin(\alpha)\\
        \sin(\alpha) & \cos(\alpha)
        \end{bmatrix}.
        $$
\end{dfn}
\begin{dfn} \label{dfn: Nieuwland 2}
    The \textbf{Nieuwland number} of $P$, denoted $\nu_P$ is the supremum over all $\nu$ such that there exists $\theta_1,\theta_2,\alpha\in[0,2\pi),\phi_1,\phi_2\in[0,\pi]$, and $x,y\in\mathbb{R}$ such that
        $$(T_{x,y}\circ R_\alpha\circ M_{\theta_1,\phi_1})(\nu P)\subset \textup{int}(M_{\theta_2,\phi_2} P).$$
\end{dfn}
\noindent
It can be shown that both definitions of the Nieuwland number are equivalent. The proof is long and unenlightening so we save it for the Appendix (\ref{sec:appendix}).
\begin{prop} \label{prop: appendix}
	The definition of the Nieuwland number of $P$ given in Definition \ref{dfn:Nieuwland 1} is equivalent to the one given in Definition \ref{dfn: Nieuwland 2}.
\end{prop}

\section{The Nieuwland Number of the Octahedron} \label{sec:dual_pass_through_condition}

Recall that the Nieuwland number of a polytope $P$ is the maximum $\nu$ so that $\nu P$ can pass through $P$. Consider the generalization of finding the largest $\nu$ such that $\nu P$ can pass through $Q$. Using \ref{proposition:equivalent_pass_through_definition}, we can phrase this as an optimization problem over the variables $(\nu, U, \eta, \delta)$. In Section \ref{subsec:dual_pass_through_condition} we give a condition that, given a feasible solution $(\nu,U,\eta,0)$ to the optimization problem corresponding to the pair $(P,Q)$, allows us to construct a feasible solution $(\nu,U^T,\eta',0)$ to the optimization problem corresponding to the pair $(Q^*,P^*)$. We then use this condition in Section \ref{subsec: numocta} to prove that the Nieuwland numbers of the cube and octahedron are the same.
\\
\subsection{The Dual Pass-Through Condition} \label{subsec:dual_pass_through_condition}
\hfill \\
Before introducing the condition mentioned above, we prove that we can take the translation factor $\delta$ to be 0 in pass through certificates for point-symmetric $P,Q$. This result was adapted from Proposition 2 of \cite{StYu23}.

\begin{prop} \label{lem:translation_factor}
	Let $P,Q$ be point-symmetric such that $P$ passes through $Q$. Then there exists some $U, \eta$ such that $P$ passes through $Q$ with $(U, \eta, 0)$ as its certificate.
\end{prop}
\begin{proof}
	Since $P$ passes through $Q$ there exists $(U,\eta,\delta)$ such that
	$$\textrm{proj}_\eta(UP+\delta)\subseteq \textrm{proj}_\eta(Q),$$ 
	$$-\textrm{proj}_\eta (UP+\delta)\subseteq -\textrm{proj}_\eta(Q).$$ 
	Since projections are linear and $P,Q$ are both point-symmetric $$\textrm{proj}_\eta(UP-\delta)\subseteq \textrm{proj}_\eta(Q).$$
	Then, $$\frac{1}{2}\left(\textrm{proj}_\eta(UP+\delta)+\textrm{proj}_\eta(UP-\delta)\right)\subseteq\frac{1}{2}\left(\textrm{proj}_\eta(Q)+\textrm{proj}_\eta(Q)\right),$$$$\textrm{proj}_\eta(UP)\subseteq \textrm{proj}_\eta(Q).$$
	As a result, $P$ passes through $Q$ with certificate $(U, \eta, 0)$.
\end{proof}
\noindent
We now present the dual pass through condition.
\begin{thm} \label{thm:dual_pass_through_cond}
	Let $0\in \textup{int}(P),\textup{int}(Q)$. Let $P=\textup{conv}\{v_1,...,v_n\}, Q=\{Ax\leq 1\}$. Suppose $P$ passes through $Q$ with certificate $(U,\eta,0)$ and there exists $\eta'\in\mathbb{R}^3$ such that for all $1 \leq k \leq n, 1 \leq i \leq m,$
	$$a_i^T(Uv_k+z_k\eta)\leq 1,$$
	where $z_k = \eta'^Tv_k$.\\\\
	Then, $Q^*$ passes through $P^*$ with the certificate $(U^T,\eta',0)$.
\end{thm}
\begin{proof}
	Since $P$ passes through $Q$ with certificate $(U,\eta,0)$, for all $1\leq k\leq n$ and $1\leq i\leq m$,
	$$1\geq a_i^T(Uv_k+\eta'^Tv_k\eta)$$
	$$=a_i^TUv_k+a_i^T\eta'^Tv_k\eta$$
	$$=(a_i^TUv_k)^T+\eta'^Tv_ka_i^T\eta$$
	$$=v_k^TU^Ta_i+(\eta'^Tv_k)^Ta_i^T\eta$$
	$$=v_k^TU^Ta_i+v_k^T\eta'a_i^T\eta$$
	$$=v_k^T(U^Ta_i+\eta'a_i^T\eta)$$
	$$=v_k^T(U^Ta_i+(a_i^T\eta)\eta').$$
	Letting $z_i'=a_i^T\eta$ gives
	$$v_k^T(U^Ta_i+z_i'\eta')\leq 1$$
	for all $1 \leq k \leq n, 1 \leq i \leq m$. We note that $P^*=\{Vx\leq1\}$ and $Q^* = \textrm{conv}\{a_1\,...,a_m\}$ where $v_k^T$ is the $k$th row of $V$. Therefore, by Proposition \ref{proposition:equivalent_pass_through_definition}, $Q^*$ passes through $P^*$ with certificate $(U^T,\eta',0)$.
\end{proof}
\begin{dfn} 
If polytopes $P$ and $Q$ along with the certificate $(U, \eta, 0)$ satisfy the hypothesis of Theorem \ref{thm:dual_pass_through_cond}, then we say they satisfy the dual pass through condition.
\end{dfn}
\hfill
\subsection{The Nieuwland Number of the Octahedron}
\label{subsec: numocta}
\hfill \\
We now demonstrate the utility of the dual pass through condition by providing a proof that the cube and octahedron have the same Nieuwland number. For the remainder of this section, we let $O=\textrm{conv}\{e_1,e_2,e_3,-e_1,-e_2,-e_3\}$ be the octahedron and $C=\{x\in\mathbb{R}^3:|x_i|\leq1, \forall 1\leq i\leq3\}$ be the cube. We let the vertices of $O$ be $v_1,...,v_6$ and the vertices of $C$ be $a_1 ,..., a_8$. Note that $O=C^*, C=O^*$. We start by proving that if an octahedron passes through itself, then the certificate satisfies the dual pass through condition.
\begin{lem} \label{lem:vo_passthrough}
	Suppose $\nu O$ passes through $O$ with certificate $(U,\eta,0)$ for some $\nu\geq1$. Then, there exists $\eta'\in\mathbb{R}^3$ such that for all $1 \leq k \leq 6, 1 \leq i \leq 8,$
	$$a_i^T(U\nu v_k+z_k\eta)\leq 1,$$ where $z_k=\eta'^T\nu v_k$.
\end{lem}
\begin{proof}
	Since $\nu O$ passes through $O$ with certificate $(U,\eta,0)$, we have that for all $1 \leq k \leq 6$, there exists $z_k \in \mathbb{R}$ such that
	$$a_i^T(U\nu v_k+z_k\eta)\leq 1.$$
	Let $v_1=e_1,v_2=e_2,v_3=e_3$. We further let $v_1=-v_4, v_2=-v_5,v_3=-v_6$. For each $a_i$ we let $a_{i'}=-a_i$. Then, $a_i^T(U\nu v_k+z_k\eta)\leq1$ if and only if $a_{i'}^T(U\nu(-v_k)-z_k\eta)\leq1$. So, we can find a solution where take $z_4=-z_1, z_5=-z_2, z_6=-z_3$.
    \\\\
    Now, set $\eta'=\frac{1}{\nu}(z_1,z_2,z_3)$. We note that for $1\leq k\leq3$ we have $$\eta'^T(\nu v_k)=\frac{1}{\nu}\nu z_k=z_k$$ and for $4\leq k\leq6$ we have $$\eta'^T(\nu v_k)=\frac{1}{\nu}(-\nu) (-z_k)=z_k.$$ Thus, for all $1\leq k\leq 6, 1\leq i \leq 8$ we have $a_i^T(U\nu v_k+z_k\eta)\leq 1$ with $z_k=\eta'^T(\nu v_k)$ as desired.
\end{proof}
\noindent
We now prove a similar result for the cube. In contrast to the previous result, we will focus solely on the optimal solution for a cube passing through a cube. The Nieuwland number for the cube is $\nu_C=\frac{3\sqrt2}{4}$ \cite{StYu23}. The following construction is taken from \cite{Ga01}.\\
\begin{center}
	\begin{tikzpicture}
        \draw[black,thick]
            (-1,-1,1)
            -- ++(2,0,0)
            -- ++(0,2,0)
            -- ++(-2,0,0)
            -- cycle;
        \draw[black,thick]
            (-1,1,1)
            -- ++(0,0,-2)
            -- ++(2,0,0)
            -- ++(0,0,2)
            -- cycle;
        \draw[black,thick]
            (1,1,-1)
            -- ++(0,-2,0)
            -- ++(0,0,2);
        \draw[black,thick,dashed]
            (-1,-1,1) -- ++ (0,0,-2);
        \draw[black,thick,dashed]
            (-1,1,-1) -- ++ (0,-2,0);
        \draw[black,thick,dashed]
            (1,-1,-1) -- ++ (-2,0,0);

        \filldraw [red, thick] (-0.5,1,-1) circle (2pt);
        \filldraw [red, thick] (1,1,0.5) circle (2pt);
        \filldraw [red, thick] (0.5,-1,1) circle (2pt);
        \filldraw [red, thick] (-1,-1,-0.5) circle (2pt);
        \filldraw[red,very thick,fill=blue,opacity=0.5]
            (-0.5,1,-1)
            -- (1,1,0.5)
            -- (0.5,-1,1)
            -- (-1,-1,-0.5)
            -- cycle;
        \node[above left,red] at (-0.5,1,-1) {$w_1$};
        \node[below right,red] at (1,1,0.5) {$w_2$};
        \node[below right,red] at (0.5,-1,1) {$w_3$};
        \node[above left,red] at (-1,-1,-0.5) {$w_4$};
        
        \pgfmathsetmacro{\myRoot}{sqrt(2)}

        \filldraw[purple, thick]
             (-0.5-\myRoot/2,1+\myRoot/4,-1+\myRoot/2)
             circle
             (2pt);
        \draw[purple,thick]
            (-0.5,1,-1) -- ++(
                -\myRoot/2,
                \myRoot/4,
                \myRoot/2
            );
        \node[above left,purple] at (-0.5-\myRoot/2,1+\myRoot/4,-1+\myRoot/2) {$u_1$};

        \filldraw[purple, thick]
             (-0.5+\myRoot/2,1-\myRoot/4,-1-\myRoot/2)
             circle
             (2pt);
        \draw[purple,thick]
            (-0.5,1,-1) -- ++(
                \myRoot/2,
                -\myRoot/4,
                -\myRoot/2
            );
        \node[above right,purple] at (-0.5+\myRoot/2,1-\myRoot/4,-1-\myRoot/2) {$u_5$};    
    \end{tikzpicture}
\end{center}
\noindent
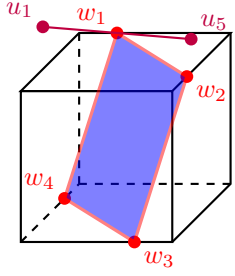
Figure 1: An illustration of a $2\times 2 \times 2$ cube, with the square defined by $w_1, w_2, w_3, w_4$ in red and $u_1,u_5$ in purple. Collectively, these define the optimal passage for a cube of side length $2 \times \frac{3\sqrt{2}}{4} = 2\nu_C$ through a cube of side length $2$.\\\\
\noindent
Consider the square with vertices by $w_1=(-\frac{1}{2},1,-1),w_2=(1,1,\frac{1}{2}),w_3=(\frac{1}{2},-1,1)$ and $w_4=(-1,-1,-\frac{1}{2})$. Indeed, we note that $$\|w_1-w_2\|=\|w_2-w_3\|=\|w_3-w_4\|=\|w_4-w_1\|=\frac{3}{2}\sqrt{2}$$ and $$\langle w_1-w_2,w_2-w_3\rangle=\langle w_2-w_3,w_3-w_4\rangle = \langle w_3-w_4,w_4-w_1\rangle = \langle w_4 -w_1, w_1-w_2\rangle=0.$$ This proves that every side length is equal and every angle is a right angle, so $w_1,w_2,w_3,w_4$ are the vertices of a square.
\\\\
Let $\eta = \left( \frac{-2}{3},\frac{1}{3}, \frac{2}{3} \right) \in S^2$. We note that $\langle\eta,w_i\rangle=0$ for each $1\leq i \leq4$. For each $1\leq i\leq 4$, we let $u_i=w_i+\frac{3\sqrt2}{4}\eta$ and $u_{i+4}=w_i-\frac{3\sqrt2}{4}\eta$. Since $\eta$ is perpendicular to each $w_i$ and $\|u_i-u_{i+4}\|$ is the same as the side length of the square given by the $w_i$, the points $u_1,...,u_8$ give the vertices of a cube $C'$ with side length $\frac{3\sqrt2}{2}$. Note that $\textrm{proj}_\eta(C')=\textrm{conv}\{w_1,...,w_4\}$ and $w_i\in C$ for each $1\leq i\leq4$. Thus, $\textrm{proj}_\eta(C')\subseteq \textrm{proj}_\eta(C)$. Since $C'$ is a cube with side length $\frac{3\sqrt2}{2}=2\nu_C$, there exists some $U$ such that $C'=\nu_C UC$, so $\nu_CC$ passes through $C$ with certificate $(U,\eta,0)$. Recall that the vertices of $C$ are given by $a_1 ,...,a_8$ and the faces by $v_1 \cdots v_6$.
\begin{lem} \label{lem:vc_passthrough}
	Consider the certificate $(U,\eta,0)$ given above. There exists $\eta'\in\mathbb{R}^3$ such that for all $1 \leq k \leq 6,1 \leq i \leq 8,$
	$$v_{k}^T(U\nu_Ca_i+z_i\eta)\leq 1,$$
    where $z_i = \eta'^T\nu_Ca_i$.
\end{lem}
\begin{proof}
 Unlike the proof for the octahedron, for this proof, we will construct $\eta'$ explicitly. We can assume, up to relabeling, $\nu_CUa_i=u_i$. We note that from the above construction we have $w_i\in C$. Then for all $1 \leq k \leq 6, 1 \leq i \leq 4$,
	$$v_{k}^T(w_i)\leq 1.$$
	Let $\eta'=-U^T\eta$. For all $1 \leq k \leq 6, 1 \leq i \leq 8,$
	$$v_{k}^T(U\nu_Ca_i+\eta'^T\nu_Ca_i\eta),$$
    $$=v_{k}^T(u_i+\eta'^T(U^{-1}u_i)\eta),$$
	$$=v_{k}^T(u_i-(U^T\eta)^T(U^{-1}u_i)\eta),$$
	$$=v_{k}^T(u_i-\eta^Tu_i\eta).$$
    Then we either have
	$$=v_{k}^T(w_i),$$
    or
    $$=v_k^T(w_{i-4}).$$
    Both of which are less than or equal to $1$.
\end{proof}
\begin{thm} \label{thm:oct_nieuwland} The Nieuwland number of the cube is equal to the Nieuwland number of the octahedron. That is,
	$$\nu_C=\nu_O.$$    
\end{thm}
\begin{proof}
	Suppose $\nu_O O$ passes through $O$ with certificate $(U,\eta,0)$. By Lemma \ref{lem:vo_passthrough} and Theorem \ref{thm:dual_pass_through_cond}, there exists $\eta'\in\mathbb{R}^3$ such that $O^*=C$ passes through $(\nu_O O)^*=\frac{1}{\nu_O}C$. Thus, by scaling by $\nu_O$, $\nu_O C$ passes through $C$. Then, by definition of $\nu_C$ we have $\nu_C \geq \nu_O$.
	\\\\
	Nieuwland's construction yields a certificate $(U,\eta,0)$ for which $\nu_C C$ passes through $C$. By Lemma \ref{lem:vc_passthrough} and Theorem \ref{thm:dual_pass_through_cond}, there exists an $\eta'$ such that $C^*=O$ passes through $(\nu_CC)^*=\frac{1}{\nu_C}O$ with certificate $(U^T,\eta',0)$. Scaling by $\nu_C$ shows us that $\nu_C O$ passes through $O$. Since $\nu_O$ is the supremum over all $\nu$ such that $\nu O$ passes through $O$ we have $\nu_O \geq \nu_C$.
    \\\\
    Thus, $\nu_C=\nu_O$ as desired.
\end{proof}
\noindent
It should be noted that this proof method is not easily modified to give proofs of equality between the Nieuwland numbers of other dual polytopes. It relies both on the simple combinatorial nature of the octahedron and the orientation of Nieuwland's solution to Rupert's problem for the cube. For the octahedron, we use the fact that it is point symmetric and has exactly 6 vertices. The important thing to note is that there are 3 unique vertices up to sign, which is equal to the dimension of the space we are working over.
\\\\
For the cube, we use the optimal solution given by Nieuwland to prove Lemma \ref{lem:vc_passthrough}. Not only that, but the proof hinges on the fact that the projection of the larger cube onto the chosen subspace lies within the smaller cube. Adapting this proof technique as is would be impossible, since no other optimal solutions to Rupert's problem are known for other polytopes. Even so, there is no reason to believe that there will exist a subspace for which the projection of the larger copy lies within the smaller copy.

\section{Nieuwland Algorithm} \label{sec:main_algorithm}

Computing exact Nieuwland numbers is difficult. Even finding approximations for small polytopes has proven to be computationally expensive. Although there are some algorithms which can theoretically compute exact values, their time complexities are on the order of $n!$ and they are impossible to run, even on polytopes with a small number of vertices \cite{StYu23}.\\\\
We develop a polynomial time algorithm for computing Nieuwland numbers. Despite being polynomial, the time complexity of the algorithm is quite large, so this algorithm is only a theoretical contribution. We hope that more techniques can be applied in combination with the core ideas here to develop tractable approaches in the future.
\\
\subsection{The Pass-Through System of Equations and Inequalities} \label{subsec:pass_through}
\hfill \\
Recall Definition \ref{def:pass_through} and Proposition \ref{proposition:equivalent_pass_through_definition}, which states that $P=\textrm{conv}(v_1, \cdots v_n)$ can pass through $Q = \{Ax \leq b\}$ if and only if
\begin{gather*}
	\exists U,\delta,\eta, z_1 \cdots z_n \in SO(3)\times \mathbb{R}^3\times\mathbb{R}^3\times \mathbb{R} \cdots \mathbb{R},\\
	\forall 1\leq k \leq n: \quad A(Uv_k+\delta+z_k\eta) \leq b.
\end{gather*}
This is a system of quadratic equations and inequalities in the variables $U, \delta, \eta$ and $z_1 \cdots z_n$. Let $\nu \geq 0$ be a free variable and consider the substitutions $Q \leftarrow P$ and $P \leftarrow \nu P$. Then, by definition, $\nu_P$ is the maximum $\nu$ such that there exist $U, \delta, \eta, z_1 \cdots z_n$ that satisfy these equations. Since all these constraints can be expressed as quadratics, finding $\nu_P$ can be formulated as a problem known in operations research as a Quadratically Constrained Quadratic Program (QCQP).\\\\
A QCQP is an optimization problem with quadratic objective function and constraints. Since algorithms for solving QCQPs exist, this definition already gives us a hypothetical way to compute the Nieuwland numbers. However, this system of equations and inequalities is too large to be optimized using standard solvers in practice. To make this approach feasible, the number of variables must be reduced. The obvious choice for this reduction are the $z_k$'s as their number grows linearly with the number of extreme points of $P$.\\\\
We construct an equivalent version of Definition \ref{def:pass_through}, which allows us to avoid the extra $z$ variables. Most of the ideas in the following proof are inspired by the Fourier-Motzkin elimination algorithm described in \cite{ConCorZam14}. 

\begin{thm} \label{thm:pass_through_simplified}
	Let $P=\textup{conv}(v_1 \cdots v_n)$ and $Q=\{Ax \leq b\}$ be polytopes and let the rows of $A$ be $a_1^T \cdots a_m^T$. Let $\eta \neq 0$, then the following are equivalent.
	\begin{itemize}
		\item $P$ passes through $Q$ with certificate $(U, \eta, \delta)$
		\item There exists a bipartition $(I_+, I_-)$ of the set $\{1 \cdots m\}$ such that the following are true.
		      \begin{enumerate}
		      	\item For all $i \in I_+, a_i^T \eta \geq 0$ and for all $j \in I_-, a_j^T \eta \leq 0$.
		      	\item For all $1 \leq k \leq n, i \in I_+, j \in I_-$
		      	      \begin{gather*}
		      	      	(b_j-a_j^T(Uv_k+\delta))(a_i^T\eta) \geq (b_i-a_i^T(Uv_k+\delta))(a_j^T\eta).
		      	      \end{gather*}
		      \end{enumerate}
	\end{itemize}
\end{thm}
\begin{proof}
	$(\Longrightarrow)$
	\\\\
	By the assumption there are $z_1, \cdots, z_n$ such that for all $1 \leq k \leq n$, we have
	\begin{gather*}
		A(Uv_k + \delta+z_k\eta) \leq b,\\
		A(Uv_k + \delta)+z_kA\eta \leq b,\\
		\begin{bmatrix}A & A\eta\end{bmatrix}\begin{bmatrix}Uv_k + \delta \\ z_k\end{bmatrix} \leq b.
	\end{gather*}
	Let $I_- = \{1 \leq i \leq m:a_i^T\eta < 0\}$ and $I_+ = \{1 \leq i \leq m:a_i^T\eta \geq 0\}$. Then $(I_-, I_+)$ is a bipartition of $\{1 \cdots m\}$.\\\\
	Now consider arbitrary $i, j, k$ such that $1 \leq k \leq n$, $i \in I_+$ and $j \in I_-$. We prove that $(I_-, I_+)$ satisfies the conclusions of the theorem. We consider two cases, $a_i^T\eta = 0$ and $a_i^T\eta > 0$. Recall that $a_j^T\eta < 0$ by the construction of $I_+$ and $I_-$.\\\\
	Consider the case where $a_i^T\eta = 0$. By the assumption, we know
	\begin{gather*}
		a_i^T(Uv_k+\delta+z_k \eta) \leq b_i.
	\end{gather*}
	Since $a_i^T \eta = 0$, this becomes
	\begin{gather*}
		a_i^T(Uv_k+\delta+z_k \eta) \leq b_i,\\
		a_i^T(Uv_k+\delta)+z_k a_i^T\eta \leq b_i,\\
		a_i^T(Uv_k+\delta) \leq b_i,\\
		0 \leq b_i-a_i^T(Uv_k+\delta).
	\end{gather*}
	Since $j \in I_-$, we have $a_j^T\eta < 0$. Thus,
	\begin{gather*}
		0 \geq (b_i-a_i^T(Uv_k+\delta))(a_j^T\eta).
	\end{gather*}
	Since we have $a_i^T\eta = 0$, $(b_j-a_j^T(Uv_k+\delta))(a_i^T\eta) = 0$. So,
	\begin{gather*}
		0 \geq (b_i-a_i^T(Uv_k+\delta))(a_j^T\eta),\\
		(b_j-a_j^T(Uv_k+\delta))(a_i^T\eta) \geq (b_i-a_i^T(Uv_k+\delta))(a_j^T\eta).\\
	\end{gather*}
	as desired.\\\\
	Now consider the case where $a_i^T\eta > 0$. From the assumptions of this direction, we have
	\begin{gather*}
		a_i^T(Uv_k+\delta) + (a_i^T\eta)z_k \leq b_i,\\
		a_j^T(Uv_k+\delta) + (a_j^T\eta)z_k \leq b_j.
	\end{gather*}
	Since $i \in I_+$ and $j \in I_-$
	\begin{gather*}
		z_k \leq \frac{b_i-a_i^T(Uv_k+\delta)}{a_i^T\eta},\\
		z_k \geq \frac{b_j-a_j^T(Uv_k+\delta)}{a_j^T\eta}.
	\end{gather*}
	Combining these yields
	\begin{gather*}
		\frac{b_j-a_j^T(Uv_k+\delta)}{a_j^T\eta} \leq \frac{b_i-a_i^T(Uv_k+\delta)}{a_i^T\eta},\\
		(b_j-a_j^T(Uv_k+\delta))(a_i^T\eta) \geq (b_i-a_i^T(Uv_k+\delta))(a_j^T\eta).
	\end{gather*}
	As desired.\\\\
	$(\Longleftarrow)$
	\\\\
	By the assumption, we have a bipartition $(I_-,I_+)$ such that for all $1 \leq k \leq n, i \in I_+, j \in I_-,$
	\begin{gather*}
		(b_j-a_j^T(Uv_k+\delta))(a_i^T\eta) \geq (b_i-a_i^T(Uv_k+\delta))(a_j^T\eta).
	\end{gather*}
    and for all $i \in I_+$ and $j \in I_-$, $a_i^T\eta \geq 0$ and $a_j^T \eta \leq 0$.\\\\
	Let $k$ be arbitrary. We want to prove that there exists a $z_k$ such that for all $1 \leq i \leq n$
	\begin{gather*}
		a_i^T(Uv_k+\delta)+z_ka_i^T\eta \leq b.
	\end{gather*}
    For each $1 \leq i \leq n$, we examine the condition above to see the constraints on $z_k$. We then show that, given the assumption, there must exist a solution.\\\\
	Let $1 \leq i \leq n$ be arbitrary. We consider three cases, $a_i^T\eta = 0, a_i^T\eta < 0$ and $a_i^T\eta  > 0$.\\\\
	If $a_i^T \eta = 0$, we then have two more cases, $i \in I_+$ and $i \in I_-$. If $i \in I_+$, the condition becomes
	\begin{gather*}
		a_i^T(Uv_k+\delta)+z_k(0) = a_i^T(Uv_k+\delta) \leq b.
	\end{gather*}
    This condition is independent of $z_k$. So, as long as it is always satisfied, it does not influence the value of $z_k$. Let $j \in I_-$ be such that $a_j^T \eta < 0$. We are guaranteed that one such $j$ must exist since $Q$ is bounded. Using the fact that $a_i^T\eta = 0$, the assumption can be simplified.
    \begin{gather*}
		(b_j-a_j^T(Uv_k+\delta))(a_i^T\eta) \geq (b_i-a_i^T(Uv_k+\delta))(a_j^T\eta),\\
        0 \geq (b_i-a_i^T(Uv_k+\delta))(a_j^T\eta).
	\end{gather*}
    Since $a_j^T\eta < 0$,
    \begin{gather*}
        0 \leq b_i-a_i^T(Uv_k+\delta),\\
        a_i^T(Uv_k+\delta) \leq b_i.
	\end{gather*}
	as a result, our assumptions guarantee this condition is always satisfied. The case where $i \in I_-$ is very similar, but with the roles of $i$ and $j$ in the previous proof reversed.\\\\
	If $a_i^T\eta > 0$, the condition becomes
	\begin{gather*}
		z_k \leq \frac{b_i-a_i^T(Uv_k+\delta)}{a_i^T\eta}.
	\end{gather*}
	If $a_i^T\eta < 0$, the condition becomes
	\begin{gather*}
		z_k \geq \frac{b_i-a_i^T(Uv_k+\delta)}{a_i^T\eta}.
	\end{gather*}
	As a result, such a $z_k$ exists if and only if for all $i \in I_+$ and $j \in I_-$
	\begin{gather*}
		\frac{b_j-a_j^T(Uv_k+\delta)}{a_j^T\eta} \leq \frac{b_i-a_i^T(Uv_k+\delta)}{a_i^T\eta},\\
		(b_j-a_j^T(Uv_k+\delta))(a_i^T\eta) \geq (b_i-a_i^T(Uv_k+\delta))(a_j^T\eta).
	\end{gather*}
    Which holds by the hypothesis.
\end{proof}
\subsection{The Nieuwland QCQP}\label{subsec:nieuwland_qcqp}
\hfill \\
The simplification of Definition \ref{def:pass_through} given by Theorem \ref{thm:pass_through_simplified} makes finding Nieuwland numbers for large polytopes substantially simpler. The core idea is still to find the largest $\nu$ so that $\nu P$ can pass through $P$. However, using the simplified version of Definition \ref{def:pass_through} requires checking all bipartitions of $\{1 \cdots m\}$. We introduce the condition $\lVert \eta \rVert^2 = 1$ to ensure that $\eta \neq 0$. To avoid cubic terms, we replace $\nu U$ with $U$ and use an adjusted version of the rotation constraints. The Nieuwland number of a convex polytope $P=\{Ax \leq b\} = \textrm{conv}(v_1 \cdots v_n)$ can be written as follows. 

\begin{gather*}
	\max_{(I_+,I_-)\textrm{ bipartition of [m]}}
	\max_{\begin{matrix}U \in M_{3 \times 3}(\mathbb{R}) \\ \eta \in \mathbb{R}^3 \\ \delta \in \mathbb{R}^3 \\ \nu \in \mathbb{R}^+ \end{matrix}}
	\nu \quad\\
	s.t.
\end{gather*}
\begin{center}
	$\lVert \eta \rVert^2 = 1$\\
	$\lVert U_1 \rVert^2 = \nu^2 \quad \lVert U_2 \rVert^2 = \nu^2 \quad U_1 \cdot U_2 = 0 \quad \nu U_3 = U_1 \times U_2$\\
	$a_i^T\eta \geq 0 \quad \forall i \in I_+$\\
	$a_i^T\eta \leq 0 \quad \forall i \in I_-$\\
	$(b_j-a_j^T(U v_k+\delta))(a_i^T\eta) \geq (b_i-a_i^T(U v_k+\delta))(a_j^T\eta) \quad \forall k =1 \cdots n, i \in I_+, j \in I_-$
\end{center}
\noindent
Computing this value requires solving $2^m$ QCQPs. While this is an improvement, it is still exponential. In order to achieve a polynomial time algorithm, we need to reduce the number of bipartitions we are checking. To accomplish this, we prove the following lemma.
\begin{lem} \label{lem:hyperplanes}
	Let $a_1 \cdots a_m \in \mathbb{R}^3$. There exists a subset $S$ of the set of bipartitions of $\{1 \cdots m\}$ with size $O(m^3)$ such that for all $\eta \in \mathbb{R}^3$, there exists $(I_-, I_+) \in S$ such that $a_i^T \eta \leq 0$ for all $i \in I_-$ and $a_i^T\eta \geq 0$ for all $i \in I_+$ .
\end{lem}
\begin{proof}
	First, we refer to a theorem stated in \cite{Ra20} which states that the planes $a_1^T\eta = 0, \cdots a_m^T \eta = 0$ divide $\mathbb{R}^3$ into $O(m^3)$ regions with non-empty interior. Next, we notice that each region corresponds to a convex cone bounded by the constraints $a_i^T \eta \leq 0$ or $a_i^T\eta \geq 0$ for each $i$. For each $i$, the choice $a_i^T \eta \leq 0$ or $a_i^T \eta \geq 0$ corresponds to choosing whether $i \in I_-$ or $I_+$. Let the set of all of these choices be $S$. More concretely, let the $O(m^3)$ regions be $R_1 \cdots R_M$. Each of these regions can be written as $R_j = \{x \in \mathbb{R}^3: a_i^T\eta (-1)^{s_{ij}} \leq 0\}$. Then for each $1 \leq j \leq M$, define $I_+^j = \{i:s_{ij} = 1\}$ and $I_-^j = \{i:s_{ij} = 0\}$. Then $S = \{(I_-^j, I_+^j):1 \leq j \leq M\}$.
\end{proof}
\noindent
The following algorithm computes all the regions given by $a_1 \cdots a_m$. 
\begin{algorithm}[H]
	\caption{Hyperplane Region Algorithm}\label{alg:HyperplaneArr}
	\begin{algorithmic}
		\Ensure $A = [a_1 \cdots a_m] \in M_{m\times3}(\mathbb{R})$
		\State $\textrm{Out} \gets [\mathbb{R}^n]$
		\Comment Vector of Polyhedrons representing regions
		\State $\textrm{Tmp} \gets []$
		\ForAll{$a_i \in A$}
		\ForAll{$P \in \textrm{Out}$}
		\State $P_1 \gets P\cap \{a_i^Tx \leq 0\}$
		\If {$P_1 \textrm{ has nonempty interior}$}
		\State $\textrm{append}(\textrm{Tmp},P_1)$
		\EndIf
		\State $P_2 \gets P\cap \{a_i^Tx \geq 0\}$
		\If {$P_2 \textrm{ has nonempty interior}$}
		\State $\textrm{append}(\textrm{Tmp},P_2)$
		\EndIf
		\EndFor
		\State $\textrm{Out} \gets \textrm{Tmp}$
		\State $\textrm{Tmp} \gets \textrm{[]}$
		\EndFor
		\State \Return $\textrm{Out}$
	\end{algorithmic}
\end{algorithm}
\noindent
In algorithm \ref{alg:HyperplaneArr}, the polytopes are represented by their half-space representation $P=\{Ax \leq b\}$. This makes finding the itersection between a given polytope $P$ and half-space $\{a_i^Tx \leq 0\}$ a constant time operation. In addition, testing if $P$ has nonempty interior reduces to finding the side length of the largest axis-aligned cube contained in $P$. This, in turn, reduces to solving a linear program, which can be done in polynomial time.\\\\
As a result, each iteration of the inner loop takes polynomial time in terms of the input. The total number of iterations of the inner loop is equal to the length of the array $\textrm{Out}$ at the time the loop starts. Suppose the inner loop starts at the $i$th iteration of the outer loop. Every element of the array $\textrm{Out}$ is a distinct region with nonempty interior bounded by the first $i-1$ elements of the input. By the proof of Lemma \ref{lem:hyperplanes}, the total number of these regions is at most $O((i-1)^3)=O(m^3)$. As a result, the number of iterations of the inner loop is always polynomial. Consequently, algorithm \ref{alg:HyperplaneArr} is a polynomial time algorithm.
\\
\subsection{Time Complexity Analysis of Algorithm \ref{alg:Nieuwland}}\label{subsec:time_complexity}
\hfill \\
Using Theorem \ref{thm:pass_through_simplified} and Lemma \ref{lem:hyperplanes}, we have reduced computing Nieuwland numbers to solving a polynomial number of QCQPs. All that remains is to upper-bound the time complexity of doing this. For simplicity's sake, we let the time complexity of solving a QCQP with $N$ variables and $M$ constraints be $\textrm{QCQP}(N,M)$. For each of the $O(m^3)$ QCQPs, we have 16 variables, namely $\eta, U, \delta$ and $\nu$. All that remains is to count the number of constraints for each QCQP.
\begin{itemize}
	\item $m$ constraints of the form $a_i^T\eta \leq 0$ or $a_i^T\eta \geq 0$.
	\item $1$ constraint $\lVert \eta \rVert^2=1$.
	\item $6$ constraints to specify that $U$ is a scaled orthogonal matrix.
	\item For each QCQP, there are at most $n\frac{m^2}{4}$ conditions of the form $(b_j-a_j^T(U v_k+\delta))(a_i^T\eta) \geq (b_i-a_i^T(U v_k+\delta))(a_j^T\eta)$.
\end{itemize}
This results in $n\frac{m^2}{4}+m+7$ constraints per QCQP, giving a time complexity of
\begin{align*}
	O \left(m^3 \textrm{QCQP}\left(16,n\frac{m^2}{4}+m+7\right) \right) \\
\end{align*}
To prove that this time complexity is polynomial, we need to upper-bound $\textrm{QCQP}\left(16,n\frac{m^2}{4}+m+7\right)$ with a polynomial. We can accomplish this using Quantifier Elimination.\\\\
First, consider the following theorem by Tarski and Seidenberg \cite{Re92}. 
\begin{thm}[Tarski-Seidenberg \cite{Re92}]
    Let $\psi$ be a logical formula of one of the following forms
    \begin{itemize}
    	\item $p(x) = 0$ or $p(x) < 0$ for $x \in \mathbb{R}^n$ and $p$ some multivariate polynomial
    	\item combinations of the above using conjunction, disjunction, and negation
    	\item combinations of the above using the logical quantifiers $\exists$ and $\forall$
    \end{itemize}
    Then there exists a formula $\psi'$ that is logically equivalent to $\psi$ and contains no quantifiers.
\end{thm}
\noindent
Formally, the set $\{x \in \mathbb{R}^n:\psi(x)=\textrm{true}\}$ is equal to the set $\{x \in \mathbb{R}^n:\psi'(x)=\textrm{true}\}$. As an example, consider the formula $\psi(a,b,c):=(\exists x: ax^2+bx+c=0)$. The equivalent formula is $\psi'(a,b,c):=\left((b^2-4ac \geq 0) \vee (a=0 \wedge b \neq 0) \vee (a = 0 \wedge b = 0 \wedge c = 0)\right)$.\\\\
For a given $\psi$, the process of computing $\psi'$ is called Quantifier Elimination. All algorithms for doing this are incredibly slow, so they will only be used to prove theoretical bounds. We use the time complexity given by Renegar in \cite{Re92}:
$$QE(N,M,d) = O\left( (Md)^{O(N)}\right)$$
\noindent
where $M$ is the number of polynomials, $d$ is the maximum degree, and $N$ is the number of variables being eliminated. (Note, this result is typically stated as depending on the bit-length of the coefficients, but we treat those as constants). We will use this bound to give a worst case upper bound on the time complexity of solving a QCQP. In practice, this will be a gross overestimate.\\\\
Consider an arbitrary QCQP.
\begin{gather*}
	\min \frac{1}{2}x^TP_0x+q_0x \quad s.t.\\
	\frac{1}{2}x^TP_ix+q_ix+r_i \leq 0 \quad \forall i
\end{gather*}
We introduce a new variable $t$ and the equation $t = \frac{1}{2}x^TP_0x+q_0x$. Then consider the set
\begin{align*}
	\left\{t \in \mathbb{R}, \exists x \in \mathbb{R}^n: t = \frac{1}{2}x^TP_0x+q_0x, \frac{1}{2}x^TP_ix+q_ix+r_i \leq 0 \quad \forall i \right\}. 
\end{align*}
By the Tarski-Seidenberg theorem, we have a quantifier free formula $\psi'$ such that
\begin{align*}
	\left\{t \in \mathbb{R}, \exists x \in \mathbb{R}^n: t = \frac{1}{2}x^TP_0x+q_0x, \frac{1}{2}x^TP_ix+q_ix+r_i \leq 0 \quad \forall i \right\} = \{t\in \mathbb{R}: \psi'(t)\}.
\end{align*}
Since $\psi'(t)$ is a system of equations and inequalities in one variable, finding the optimal value is achievable in polynomial time. However, in practice, ``the minimum $t \in \mathbb{R}$ that satisfies $\psi'(t)$'' is an acceptable answer, as it avoids imprecision. This allows us to optimize a QCQP with $n$ variables and $m$ constraints in time complexity.
\begin{gather*}
	QCQP(N,M)=QE(N+1,M+1,2)=O\left((2M+2)^{O(N)} \right)
\end{gather*}
This gives us the following as our final time complexity. 
\begin{gather*}
	O \left(m^3 \textrm{QCQP}\left(16,n\frac{m^2}{4}+m+7\right) \right)\\
	= O\left(m^3\textrm{QE}\left(17,n\frac{m^2}{4}+m+8, 2\right) \right)\\
	= O\left(m^3\left(n\frac{m^2}{2}+2m+16\right)^{O(1)}\right)
\end{gather*}
which is polynomial.

\begin{thm}
	There exists an algorithm that can determine the Nieuwland number of a convex polytope $P \subseteq \mathbb{R}^3$ given by its vertices in polynomial time.
\end{thm}
\begin{proof}
	See Algorithm \ref{alg:Nieuwland}.
\end{proof}

\begin{algorithm}[H]
	\caption{Nieuwland Number Algorithm}\label{alg:Nieuwland}
	\begin{algorithmic}
		\Ensure $P = \textrm{conv}(v_1 \cdots v_n)=\{Ax \leq b\}$
		\State Regions = \Call{HyperplaneArr}{$A$}
		\State $\textrm{Max} = 0$
		\ForAll{$R \in \textrm{Regions}$}
		\State Compute $(I_+, I_-)$ from $R$
		\State $\textrm{Model} \gets \textrm{QCQP}(\eta, U, \delta, \nu)$
		\ForAll{$i \in I_+$}
		\State Model.addConstraint($a_i^T \eta \geq 0$)
		\EndFor
		\ForAll{$i \in I_-$}
		\State Model.addConstraint($a_i^T \eta \leq 0$)
		\EndFor
        \State Model.addConstraint($\lVert \eta \rVert^2 = 1$)
		\State Model.addConstraint($\lVert U_1 \rVert^2 = \nu^2$)
        \State Model.addConstraint($\lVert U_2 \rVert^2 = \nu^2$)
        \State Model.addConstraint($U_1 \cdot U_2 = 0$)
        \State Model.addConstraint($\nu U_3 = U_1 \times U_2$)
		\For{$k = 1 \cdots n$}
		\ForAll{$i \in I_+$}
		\ForAll{$j \in I_-$}
		\State Model.addConstraint($(b_j-a_j^T(Uv_k+\delta))(a_i^T\eta) \geq (b_i-a_i^T(Uv_k+\delta))(a_j^T\eta)$)
		\EndFor
		\EndFor
		\EndFor
		\State $\textrm{Max} = \textrm{max(Max,Model.optimize())}$
		\EndFor
		\State \Return Max
	\end{algorithmic}
\end{algorithm}

\section{Conclusion and Future Work}\label{sec:conclusion}

This work mainly focused on answering the question raised by Steininger and Yurkevich in \cite{StYu23} about the cube and the octahedron. However, many of the ideas here were originally intended to address a much larger question. Do all point-symmetric convex polytopes have the same Nieuwland number as their polar dual? The ideas presented in this paper provide a starting point for exploring this and many other questions related to Nieuwland numbers and Rupert's property.\\\\
There are also numerous shapes for which the exact Nieuwland number is unknown. We believe that some combination of the ideas presented here and in \cite{StYu25} will be able to solve many of these problems. In particular, having a way to solve these QCQPs exactly could allow us to finally determine whether the rhombicosidodecahedron has Rupert's property.

\section{Appendix}\label{sec:appendix}

In this appendix, we will present the proof that our definition of Nieuwland numbers are equivalent to those given in \cite{StYu23}. We start with the following proposition.
\\
\begin{prop}
Let $\eta\in S^2$ and $U\in SO(3)$. Then, $$\textup{proj}_\eta\circ U=U\circ \textup{proj}_{U^{-1}\eta}.$$
\begin{proof}
    To prove equality of linear operators, it suffices to check they have the same output on any input. Let $x\in\mathbb{R}^3$.
    $$\textrm{proj}_\eta\circ(Ux)$$
    $$=Ux-\langle Ux,\eta\rangle\eta$$
    $$=Ux-\langle x,U^T\eta\rangle\eta$$
    $$=U(x-\langle x,U^{-1}\eta\rangle U^{-1}\eta)$$
    $$=U(\textrm{proj}_{U^{-1}\eta}x).$$
\end{proof}
\end{prop}
\begin{lem} \label{lem:pass_through_equiv}
    There exist $(x,y)\in\mathbb{R}^2,\theta_1,\theta_2,\alpha\in[0,2\pi)$ and $\varphi_1,\varphi_2\in[0,\pi]$ such that $$T_{x,y}\circ R_\alpha\circ M_{\theta_1,\phi_1}(P)\subseteq M_{\theta_2,\phi_2}(Q)$$ if and only if $P$ passes through $Q$.
\begin{proof}
    We begin by noting that $M_{\theta,\phi}=\begin{bmatrix}
		-\sin(\theta) & \cos(\theta) & 0\\
		-\cos(\theta)\cos(\phi) & -\sin(\theta)\cos(\phi) & \sin(\phi)
		\end{bmatrix}$ consists of 2 rows of a rotation matrix $U_{\theta,\phi}$ where the third row is given by $X(\theta,\phi)=(\cos(\theta)\sin(\phi),\sin(\theta)\sin(\phi),\cos(\phi))$.
\\\\
$(\Longrightarrow)$
\\\\
Let $(x,y)\in\mathbb{R}^2,\theta_1,\theta_2,\alpha\in[0,2\pi)$ and $\varphi_1,\varphi_2\in[0,\pi]$ such that $$T_{x,y}\circ R_\alpha\circ M_{\theta_1,\phi_1}(P)\subseteq M_{\theta_2,\phi_2}(Q).$$ Let 
$A=\begin{bmatrix} 1 & 0 & 0 \\ 0 & 1 & 0 \end{bmatrix}$. Then, $M_{\theta,\phi}=AU_{\theta,\phi}$. Note that $A^TA=\begin{bmatrix} 1 & 0 & 0\\ 0 & 1 & 0 \\ 0& 0 & 0 \end{bmatrix}=\textrm{proj}_{e_3}.$ We now quickly explore how $A^T$ interacts with $T_{x,y}$ and $R_\alpha$. First, we show $A^T\circ T_{x,y}=T_{x,y,0}\circ A^T$ where $T_{x,y,0}$ is the translation by $(x,y,0)^T$ in $\mathbb{R}^3$. Let $v\in\mathbb{R}^2$. Then,
$$(A^T\circ T_{x,y})(v)=A^T(v_1+x,v_2+y)=(v_1+x,v_2+y,0)$$
$$=T_{x,y,0}(v_1,v_2,0)=(T_{x,y,0}\circ A^T)(v).$$
Also, since $(x,y,0)$ lines on the plane orthogonal to $e_3$, we have $T_{x,y,0}\circ \textrm{proj}_{e_3}=\textrm{proj}_{e_3}\circ T_{x,y,0}$.
Now we show that $A^T\circ R_\alpha=R_\alpha'\circ A^T$ where $R_\alpha'=\begin{bmatrix} R_\alpha & 0 \\ 0 & 1 \end{bmatrix}$.
$$A^T\circ R_\alpha=\begin{bmatrix} R_\alpha \\ 0\end{bmatrix},$$
$$R_\alpha '\circ A^T=\begin{bmatrix} R_\alpha \\ 0\end{bmatrix}=A^T\circ R_\alpha.$$
A similar calculation will reveal that $R_\alpha'\circ \textrm{proj}_{e_3}=\textrm{proj}_{e_3}\circ R_\alpha'.$
The rest of this proof is simple algebra. Since $T_{x,y}\circ R_\alpha\circ M_{\theta_1,\phi_1}(P)\subseteq M_{\theta_2,\phi_2}(Q)$ we have $$A^T\circ T_{x,y}\circ R_\alpha\circ M_{\theta_1,\phi_1}(P)\subseteq A^T\circ M_{\theta_2,\phi_2}(Q).$$ Applying the facts above gives $$T_{x,y,0}\circ R'_\alpha\circ A^T\circ A\circ U_{\theta_1,\phi_1}(P)\subseteq A^T\circ A\circ U_{\theta_2,\varphi_2}(Q).$$ Since $A^T\circ A=\textrm{proj}_{e_3}$ we have $$T_{x,y,0}\circ R'_\alpha\circ \textrm{proj}_{e_3}\circ U_{\theta_1,\phi_1}(P)\subseteq \textrm{proj}_{e_3} \circ U_{\theta_2,\varphi_2}(Q).$$ Commuting $R_\alpha'$ then $T_{x,y,0}$ with $\textrm{proj}_{e_3}$ gives $$\textrm{proj}_{e_3}\circ T_{x,y,0}\circ R'_\alpha \circ U_{\theta_1,\phi_1}(P)\subseteq \textrm{proj}_{e_3} \circ U_{\theta_2,\varphi_2}(Q).$$ Recalling the previous proposition gives $$U_{\theta_2,\phi_2}\circ \textrm{proj}_{U_{\theta_2,\phi_2}^{-1}e_3}\circ U_{\theta_2,\phi_2}^{-1}\circ T_{x,y,0}\circ R'_\alpha \circ U_{\theta_1,\phi_1}(P)\subseteq U_{\theta_2,\phi_2}\circ \textrm{proj}_{U_{\theta_2,\phi_2}^{-1}e_3} (Q).$$ By applying $U_{\theta_2,\phi_2}^-1$ to both sides and commuting $U_{\theta_2,\phi_2}^{-1}$ and $T_{x,y,0}$ we have
$$\textrm{proj}_{U_{\theta_2,\phi_2}^{-1}e_3}\circ T_{U_{\theta_2,\phi_2}^{-1}(x,y,0)}\circ U_{\theta_2,\phi_2}^{-1}\circ R'_\alpha \circ U_{\theta_1,\phi_1}(P)\subseteq \textrm{proj}_{U_{\theta_2,\phi_2}^{-1}e_3} (Q).$$ Finally, we set $\eta=U_{\theta_2,\phi_2}^{-1}e_3,U=U_{\theta_2,\phi_2}^{-1}\circ R'_\alpha \circ U_{\theta_1,\phi_1},\delta=U_{\theta_2,\phi_2}^{-1}(x,y,0)$ which gives $\textrm{proj}_\eta(UP+\delta)\subseteq \textrm{proj}_\eta(Q)$ as desired.
\\\\
($\Longleftarrow$)
\\\\
Let $P$ pass through $Q$ with certificate $(U,\eta,\delta)$. We define $U_{\theta,\phi}$ as we did above. We will try to use the equations for $(U,\eta,\delta)$ given above to try and derive appropriate choices for $(\theta_1,\theta_2,\phi_1,\phi_2,\alpha,x,y)$. We start by finding a choice for $(\theta_2,\phi_2)$. Above we chose $\eta=U^{-1}_{\theta_2,\phi_2}e_3$. We note that $X(\theta,\phi)$ parametrizes the sphere. Thus, there exists $\theta_2\in[0,2\pi),\phi_2\in[0,\pi]$ such that $\eta=X(\theta_2,\phi_2)$. Since $U^{-1}_{\theta_2,\phi_2}=U^T_{\theta_2,\phi_2}$ and the third row of $U_{\theta_2,\phi_2}$ is $X(\theta_2,\phi_2)$ this gives $\eta=U^{-1}_{\theta_2,\phi_2}e_3$.
\\\\
Next we find an appropriate choice for $(x,y)$. Above we had $\delta=U^{-1}_{\theta_2,\phi_2}(x,y,0)$. We can't guarantee that $U_{\theta_2,\phi_2}(\delta)=(x,y,0)$ but we can choose $x,y$ to be the first and second coordinates of $U_{\theta_2,\phi_2}(\delta)$.
\\\\
Now we will choose $(\theta_1,\phi_1,\alpha)$ such that $U=U_{\theta_2,\phi_2}^{-1}\circ R'_\alpha \circ U_{\theta_1,\phi_1}$. By rearranging we have $$U_{\theta_2,\phi_2}\circ U\circ  U_{\theta_1,\phi_1}^{-1}=R'_\alpha.$$ Recall that $R'_\alpha$ is a block diagonal matrix $R'_\alpha=\begin{bmatrix} R_\alpha & 0 \\ 0 & 1 \end{bmatrix}$. In particular, $R'_\alpha$ is a special orthogonal matrix that fixes $e_3$. Such matrices are all of the form $\begin{bmatrix} R & 0 \\ 0 & 1 \end{bmatrix}$ for some 2 dimensional rotation matrix $R$. Since all such $R$ are parametrized by $R_\alpha, \alpha\in[0,2\pi)$, it suffices to find $\theta_1,\phi_1$ such that $U_{\theta_2,\phi_2}\circ U\circ  U_{\theta_1,\phi_1}^{-1}$ fixes $e_3$.
\\\\
Let $x=(U^{-1}\circ U_{\theta_2,\phi_2}^{-1})e_3$. Since orthogonal matrices are isometries, $x\in S^2$. Thus, we can find $\theta_1\in [0,2\pi), \phi_1\in[0,\pi]$ such that $X(\theta_1,\phi_1)=x$. This gives $$U_{\theta_1,\phi_1}^Te_3=X(\theta_1,\phi_1)=x=(U^{-1}\circ U_{\theta_2,\phi_2}^{-1})e_3.$$ Rearranging and noticing that $U_{\theta_1,\phi_1}^T=U_{\theta_1,\phi_1}^{-1}$ gives $$(U_{\theta_2,\phi_2}\circ U\circ  U_{\theta_1,\phi_1}^{-1})e_3=e_3$$ as we wanted. Thus, we can find $\theta_1,\alpha\in[0,2\pi),\phi_1\in[0,\pi]$ such that $U=U_{\theta_2,\phi_2}^{-1}\circ R'_\alpha \circ U_{\theta_1,\phi_1}$.
\\\\
All that remains is to run through the proof of the forwards direction but in reverse. We have $\textrm{proj}_\eta(UP+\delta)\subseteq \textrm{proj}_\eta(Q)$. Expanding our equations for $(U,\eta)$ gives $$\textrm{proj}_{U_{\theta_2,\phi_2}^{-1}e_3}(U_{\theta_2,\phi_2}^{-1}\circ R'_\alpha \circ U_{\theta_1,\phi_1}(P)+\delta)\subseteq \textrm{proj}_{U_{\theta_2,\phi_2}^{-1}e_3}(Q).$$ By the previous proposition we have $$U_{\theta_2,\phi_2}^{-1}\circ \textrm{proj}_{e_3}( R'_\alpha \circ U_{\theta_1,\phi_1}(P)+U_{\theta_2,\phi_2}\delta)\subseteq U_{\theta_2,\phi_2}^{-1}\circ \textrm{proj}_{e_3}(U_{\theta_2,\phi_2}Q).$$ Applying the rotation $U_{\theta_2,\phi_2}$ to both sides gives $$\textrm{proj}_{e_3}( R'_\alpha \circ U_{\theta_1,\phi_1}(P)+U_{\theta_2,\phi_2}\delta)\subseteq \textrm{proj}_{e_3}(U_{\theta_2,\phi_2}Q).$$ Since $(x,y)$ are the first 2 coordinates of $U_{\theta_2,\phi_2}\delta$ and $R_\alpha'$ commutes with $\textrm{proj}_{e_3}$ we have $$R'_\alpha \circ \textrm{proj}_{e_3}(U_{\theta_1,\phi_1}(P))+(x,y,0)\subseteq \textrm{proj}_{e_3}(U_{\theta_2,\phi_2}Q).$$ Expanding $\textrm{proj}_{e_3}=A^T\circ A$ and recalling $A^T(x,y)=(x,y,0),R'_\alpha\circ A^T=A^T\circ R_\alpha ,A\circ U_{\theta,\phi}=M_{\theta,\phi}$ we obtain $$A^T(R_\alpha \circ M_{\theta_1,\phi_1}(P)+(x,y))\subseteq A^T(M_{\theta_2,\phi_2}(Q)).$$ Finally, we note that $A^T$ is injective which gives $$(T_{x,y}\circ R_\alpha\circ M_{\theta_1,\phi_1})(P)\subseteq M_{\theta_2,\phi_2}(Q),$$ as desired.
\end{proof}
\end{lem}
\noindent
With Lemma \ref{lem:pass_through_equiv} we are now able to give a proof of Proposition \ref{prop: appendix}.
\begin{proof}
Let $\nu_1$ be the Nieuwland number of $P$ as given by Definition \ref{dfn:Nieuwland 1} and $\nu_2$ be the Nieuwland number of $P$ as given by Definition \ref{dfn: Nieuwland 2}. We start by noting that if
$$(T_{x,y}\circ R_\alpha\circ M_{\theta_1,\phi_1})(\nu P)\subseteq \textup{int}(M_{\theta_2,\phi_2}P)$$
then,
$$(T_{x,y}\circ R_\alpha\circ M_{\theta_1,\phi_1})(\nu P)\subseteq M_{\theta_2,\phi_2}P.$$ By Lemma \ref{lem:pass_through_equiv}, for any such $\nu$, $\nu P$ passes through $\nu$. Since $\nu_2$ is the supremum over such $\nu$ and the set of all such $\nu$ is a subset of the $\nu$ for which $\nu P$ passes through $P$, we have $\nu_1 \geq \nu_2$.
\\\\
Similarly, since $\nu_1 P$ passes through $P$, there exist $x,y\in\mathbb{R},\theta_1,\theta_2,\alpha\in[0,2\pi),\phi_1,\phi_2\in[0,\pi]$ such that $$(T_{x,y}\circ R_\alpha\circ M_{\theta_1,\phi_1})(\nu_1 P)\subseteq M_{\theta_2,\phi_2}P.$$ Fix $0< \epsilon < 1$ and let $u\in \textrm{int}(P)$. Then,
$$(T_{x,y}\circ R_\alpha\circ M_{\theta_1,\phi_1})(\nu_1 (P+u-u))\subseteq M_{\theta_2,\phi_2}(P+u-u).$$ Rearranging yields, $$(T_{x,y}\circ R_\alpha\circ M_{\theta_1,\phi_1})(\nu_1 (P-u))+(R_\alpha\circ M_{\theta_1,\phi_1})(\nu_1 u)-M_{\theta_2,\phi_2}(u)\subseteq M_{\theta_2,\phi_2}(P-u),$$ which can be rewritten as $$(T_{(x,y)+(R_\alpha\circ M_{\theta_1,\phi_1})(\nu_1 u)-M_{\theta_2,\phi_2}(u)}\circ R_\alpha\circ M_{\theta_1,\phi_1})(\nu_1 (P-u))\subseteq M_{\theta_2,\phi_2}(P-u).$$ We note that since $u\in \textrm{int}(P)$ we have $0\in \textrm{int}(P-u)$. Thus, for our fixed $\epsilon$ we have
$$\varepsilon(T_{(x,y)+(R_\alpha\circ M_{\theta_1,\phi_1})(\nu_1 u)-M_{\theta_2,\phi_2}(u)}\circ R_\alpha\circ M_{\theta_1,\phi_1})(\nu_1 (P-u))\subseteq \textrm{int}(M_{\theta_2,\phi_2}(P-u)).$$ Commuting the scalar $\epsilon$ gives
$$(T_{\epsilon((x,y)+(R_\alpha\circ M_{\theta_1,\phi_1})(\nu_1 u)-M_{\theta_2,\phi_2}(u))}\circ R_\alpha\circ M_{\theta_1,\phi_1})(\epsilon\nu_1 (P-u))\subseteq \textrm{int}(M_{\theta_2,\phi_2}(P-u)).$$ We can again rearrange to obtain $$(T_{\epsilon((x,y)+(R_\alpha\circ M_{\theta_1,\phi_1})(\nu_1 u)-M_{\theta_2,\phi_2}(u))}\circ R_\alpha\circ M_{\theta_1,\phi_1})(\epsilon\nu_1 P)-(R_\alpha\circ M_{\theta_1,\phi_1})(\epsilon\nu_1 u)+M_{\theta_2,\phi_2}(u)\subseteq \textrm{int}(M_{\theta_2,\phi_2}P),$$ which can be rewritten as $$(T_{\epsilon(x,y)+(1-\varepsilon)M_{\theta_2,\phi_2}(u)}\circ R_\alpha\circ M_{\theta_1,\phi_1})(\epsilon\nu_1 P)\subseteq \textrm{int}(M_{\theta_2,\phi_2}P).$$ This implies that $\nu_2\geq \varepsilon \nu_1$ for all choices of $0<\epsilon<1$. Taking $\varepsilon\to1^-$ gives $\nu_2\geq\nu_1$. Thus, $\nu_1=\nu_2$ as desired.
\end{proof}

\bibliographystyle{plain}
\bibliography{bib} 
                            
\end{document}